\documentclass[11pt,reqno]{amsart}
\usepackage{amsbsy,amsfonts,amsmath,amssymb,amscd,amsthm}
\usepackage{mathrsfs,float}
\usepackage[mathcal]{euscript}
\usepackage{subfigure,enumitem,graphicx,epstopdf}
\usepackage[a4paper,text={6.1in,8in},centering]{geometry}
\usepackage{pxfonts,epstopdf}
\usepackage[colorlinks=true,linkcolor=blue,citecolor=green]{hyperref}
\usepackage{srcltx}
\usepackage[margin=20pt,font=small,labelfont=bf,textfont=it]{caption}

\hypersetup{linktocpage}

\small\normalsize
\theoremstyle{plain}
\newtheorem{mainthm}{Theorem}
\newtheorem{maincor}{Corollary}
\newtheorem{thm}{Theorem}[section]
\newtheorem{lem}[thm]{Lemma}
\newtheorem{pro}[thm]{Proposition}
\newtheorem{prop}[thm]{Proposition}
\newtheorem*{thm*}{Theorem}
\newtheorem{cor}[thm]{Corollary}

\newtheorem{defi}[thm]{Definition}
\newtheorem{rem}[thm]{Remark}
\newtheorem{example}[thm]{Example}

\theoremstyle{definition}

\newtheorem{claim}{Claim}

 \DeclareMathOperator{\diam}{diam}
\DeclareMathOperator{\Ls}{Ls}
\newcommand{\eqdef}{\stackrel{\scriptscriptstyle\rm def}{=}}

\renewcommand{\theclaim}{\arabic{section}.\arabic{thm}.\arabic{claim}}
\setlist[enumerate,1]{label=(\arabic*)}
\setlist[enumerate,2]{label=(\alph*)}

\begin{document}

\title[Chaos game of IFSs]{On the chaos game of iterated function systems}

\author[P.~G.~Barrientos --- F.~H.~Ghane --- D.~Malicet --- A.~Sarizadeh]{Pablo G.~Barrientos --- Fatemeh H.~Ghane \\ Dominique Malicet --- Aliasghar Sarizadeh}

\address{\textsc{Pablo G.~Barrientos}\\
Instituto de Matem\'atica e Estat\'istica\\
Universidade Federal Fluminense\\
Rua M\'ario Santos Braga s/n - Campus Valonguinhos, Niter\'oi,
Brazil} \email{pgbarrientos@id.uff.br}

\address{\textsc{Fatemeh H.~Ghane}\\
Department of Mathematics\\
Ferdowsi University of Mashhad\\
Mashhad, Iran}
\email{ghane@math.um.ac.ir}

\address{\textsc{Dominique Malicet}\\
Instituto de Matem\'atica e Estat\'istica\\
Universidade do Estado do Rio de Janeiro\\
S\~ao Francisco Xavier, 524 - Pavilh\~ao Reitor Jo\~ao Lyra Filho,
Rio de Janeiro, Brazil} \email{malicet@mat.puc-rio.br}

\address{\textsc{Aliasghar Sarizadeh}\\
Department of Mathematics\\
Ilam University\\
Ilam, Iran}
\email{a.sarizadeh@mail.ilam.ac.ir}
\subjclass[2010]{Primary: 37C05, 37C20; Secondary: 37E10\hfill}
\thanks{The first author was partially supported by
Ministerio de Ciencia e Innovaci\'on project MTM2014-56953-P}
~\vspace{1cm}
\begin{abstract} \vspace{1cm}
Every quasi-attractor of an iterated function system (IFS) of
continuous functions on a first-countable Hausdorff topological
space is renderable by the probabilistic chaos game. By contrast,
we prove that the backward minimality is a necessary condition to
get the deterministic chaos game. As a consequence, we obtain that
an IFS of homeomorphisms of the circle is renderable by the
deterministic chaos game if and only if it is forward and backward
minimal. This result provides examples of attractors (a forward
but no backward minimal IFS on the circle) that are not renderable
by the deterministic chaos game. We also prove that every
well-fibred quasi-attractor is renderable by the deterministic
chaos game as well as quasi-attractors of both, symmetric and
non-expansive IFSs.
\end{abstract}
\maketitle
\section{Introduction}
Within fractal geometry, iterated function systems (IFSs) provide
a method for both generating and characterizing fractal images. An
\emph{iterated function system} (IFS) can also be thought of as a
finite collection of functions which can be applied successively
in any order. Attractors of this kind of systems are self-similar
compact sets which draw any iteration of any point in an open
neighborhood of itself.

There are two methods for generating the attractor: the
\emph{deterministic} algorithm, in which all the transformations
are applied simultaneously, and the \emph{random} algorithm, in
which the transformations are applied one at a time in random
order following a probability. The \emph{chaos game}, popularized
by Barnsley~\cite{B88}, is the simple algorithm implementing the
random method. We have two different forms to run the chaos game.
One involves taking a starting point and then choose randomly the
transformation on each iteration accordingly to the assigned
probabilities. The latter starts by choosing a random order
iteration and then applying this orbital branch anywhere in the
basin of attraction. The first form of implementation is called
\emph{probabilistic chaos game}~\cite{BV11,BLR15}. The second
implementation is called~\emph{deterministic chaos game} (also
called~\emph{disjunctive chaos game})~\cite{L14,BL14,BFS14}.

 According to~\cite{BLR15}, every attractor of an
IFS of continuous maps on a first- countable Hausdorff topological
space is renderable by the probabilistic chaos game. By contract,
we will see that this is not the case of the deterministic chaos
game. Namely, we will provide necessary and sufficient conditions
to get the deterministic chaos game. As an application we will
obtain that an IFS of homeomorphisms of the circle is renderable
by the deterministic chaos game if and only if it is forward and
backward minimal which provides examples of attractors that are
not renderable by the deterministic chaos game.

\subsection{Iterated function systems}
Let $X$ be a Hausdorff topological space. We consider a finite set
$\mathscr{F}=\{f_1,\dots f_k\}$ of continuous functions from $X$
to itself. Associated with this set $\mathscr{F}$ we define the
\emph{semigroup} $\Gamma=\Gamma_\mathscr{F}$ generated by these
functions, the \emph{Hutchinson operator} $F=F_\mathscr{F}$ on the
hyperspace $\mathscr{H}(X)$ of the non-empty compact subsets of
$X$
$$
   F: \mathscr{H}(X) \to \mathscr{H}(X), \qquad F(A)=\bigcup_{i=1}^k
   f_i(A)
$$
and the \emph{skew-product} $\Phi=\Phi_\mathscr{F}$ on the product
space of $\Omega=\{1,\dots,k\}^\mathbb{N}$ and $X$
$$
\Phi: \Omega\times X \to \Omega \times X, \qquad
\Phi(\omega,x)=(\sigma(\omega),f_{\omega_1}(x)),
$$
where $\omega=\omega_1\omega_2\dots \in \Omega$ and
$\sigma:\Omega\to \Omega$ is the lateral shift map. The action of
the semigroup $\Gamma$ on $X$ is called the \emph{iterated function
system} generated by $f_1,\dots,f_k$ (or, by the family
$\mathscr{F}$ for short). Finally, given
$\omega=\omega_1\omega_2\dots\in \Omega$ and $x\in X$,
$$
  O^+_\omega(x)=\{f^n_\omega(x): n\in \mathbb{N}\}
  \quad \text{where} \ \
  f^n_\omega\eqdef f_{\omega_n}\circ \cdots \circ f_{\omega_1}
  \quad \text{for every $n\in\mathbb{N}$,}
$$
are called, respectively, the \emph{$\omega$-fiberwise orbit} of
$x$ and the \emph{orbital branch} corresponding to $\omega$ (or
the IFS-iteration driving by the sequence $\omega$). We introduce
now some different notions of invariant and minimal sets and after
that we give the definition of an attractor. In what follows $A$
denotes a  closed subset of $X$.

\subsection{Invariant and minimal sets}
We say that $A$ is a \emph{forward invariant} set if $f(A)\subset
A$ for all $f\in \Gamma$.
We also say that $A$ is a \emph{self-similar} set if
$$
A=f_1(A)\cup \dots \cup f_k(A).
$$
Notice that a minimal set regarding to the inclusion of forward
invariant non-empty (closed) sets is always a self-similar set. We
simply call it  as \emph{forward invariant minimal} set. By
extension, we say that the IFS is \emph{forward minimal} if the
unique forward invariant non-empty closed set is the whole space.
It is not difficult to see that forward minimality is equivalent
to density of any $\Gamma$-orbit. That is,  $A$ is a forward
invariant minimal set if and only if  $A$ coincides with the
closure of $\Gamma$-orbit $ \Gamma(x)\eqdef \{g(x): g\in \Gamma\}$
for all $x\in A$. Similarly, we will say that $A$ is a
\emph{forward minimal set} if $A$ is contained in the closure of
$\Gamma(x)$  for all $x\in A$. Thus, \emph{forward minimal
self-similar sets} are forward invariant minimal sets and
viceversa.

\begin{defi}
We say that $A$ is a \emph{quasi-attractor} of the IFS generated
by $\mathscr{F}$ if it is a forward minimal self-similar compact
set, i.e., if
$
    A \in \mathscr{H}(X), \ \ F(A)=A \ \ \text{and} \ \
    A=\overline{\Gamma(x)} \ \text{for all $x\in A$.}
$
\end{defi}

 Finally, notice that, as a straightforward
application of Zorn's lemma, every IFS on a compact space has a
quasi-attractor.

\subsection{Attractors} We introduce the notion of attractor
following~\cite{BV11,BV13,BLW14,BLR15}. To accomplish this, we
need to define first the pointwise basin of attraction.
\vspace{0.1cm}

Given a compact set $K$ of $X$, the \emph{$\Ls$-limit set} (also
called $\omega$-limit set or  topological upper limit set) of $K$
for $F$ is the set
$$
   \Ls F^n(K)  \eqdef \bigcap_{m\in\mathbb{N}} \overline{\bigcup_{n\geq m}
   F^n(K)}.
$$
Observe that $\Ls F^n(S)$ is always closed. However it can be
non-compact.  Now, let $A$ be a compact set. The \emph{pointwise
basin of $\Ls$-attraction} of $A$ for $F$  is defined to be the
set
$$
  \mathcal{B}^*_p(A)\eqdef \{x\in X: \, \Ls F^n(\{x\})= A \}. \vspace{0.1cm}
$$
Similarly,
the \emph{pointwise basin of
\emph{Vietoris}-attraction} for $F$ is the set
$$
\mathcal{B}_p(A)
\eqdef \{x\in X: \,  \lim_{n\to \infty } F^n(\{x\}) =A \}.
$$  The
convergence here is with respect to the Vietoris topology, or
equivalently, in the metric space case, with respect to the
Hausdorff metric~\cite[pp.~66--69]{S98}. Moreover, it is not
difficult to show that $\mathcal{B}_p(A) \subset
\mathcal{B}_p^*(A)$.

A compact set $A$ is a \emph{pointwise attractor}
if there is an open set $U$ of $X$ such that $A\subset U \subset
\mathcal{B}_p(A)$. A slightly stronger notion of an attractor is
the following. $A$  is said to be a \emph{strict attractor}
if there is an open neighborhood $U$ of $A$ such that
$\lim_{n\to\infty} F^n(K) =  A$ 
in the
Vietoris topology for all compact sets $K\subset U$.
We denote by  $\mathcal{B}(A)$  the \emph{basin}
of the strict attractor. That is, the union of all open
neighborhoods $U$ of $A$ such that
the above convergence holds.

We remark that it is usual to include in the definition of
attractor that $F(A)=A$ (cf.~\cite[Def.~2.2]{BV13}). Under our
mild assumptions on $X$, it is unknown the continuity of the
Hutchinson operator (see~\cite{BL15}) and thus it is not, a
priori, clear that $A$ is a self-similar ($F$-invariant) set for
the IFS. Nevertheless, the following result proves that any
attractor must be a quasi-attractor and, in the case of a strict
attractor, must attract any compact set in the basin of attraction
which, a priori, is also not clear from the definition.
\begin{mainthm}
\label{thmA-attractor} Consider the IFS generated by $\mathscr{F}$
and a compact subset $A$.
\begin{enumerate}[itemsep=0.2cm]
\item \label{item1-thmA-attractor} $A$ is a quasi-attractor if and
only if $A\subset \mathcal{B}^*_p(A)$. Moreover, in this case,
$$
    A=\Ls F^n(K) \quad \text{for all non-empty compact sets $K\subset A$;}
$$
\item \label{item1-thmA-attractor3}
If $A$ is a pointwise attractor, it is a quasi-attractor and
$\mathcal{B}_p(A) = \mathcal{B}^*_p(A);$
\item \label{item2-thmA-attractor}
If $A$ is a strict attractor, it is a pointwise attractor,
$\mathcal{B}(A) = \mathcal{B}_p(A)$ and for every non-empty
compact set $K \subset \mathcal{B}(A)$
$$
  \lim_{n\to\infty} F^n(K)= A  \ \ \text{in the Vietoris topology. }
$$
\end{enumerate}
\end{mainthm}
\vspace{0.1cm}

Another notion of the ``attractor'' of an IFS is the concept of
\emph{semi-attractor} introduced by Lasota and Myjak
in~\cite{LM96}. Semi-attractors, sometimes called semi-fractals,
are the smallest (unique) forward invariant set defined by means
of the Kuratowski topological limits. We refer to~\cite{LM99,LMT04}
for a precise definition. Thus, these sets are also forward
invariant self-similar sets (in particular closed sets) but in
contracts with strict/pointwise attractors or quasi-attractors,
semi-attractors can be non-compact. \vspace{0.1cm}

Examples of pointwise attractors that are not strict attractors
can be find in~\cite{BLR15}. Also, one easily can construct
quasi-attractors of IFSs that are neither attractors nor
semi-attractor. A simple example is provided by an IFS generated
by a minimal map $f$ (for instance a rotation of the circle with
irrational rotation number). The whole space $A$ is the unique
non-empty forward invariant closed set but it is not the limit in
the Hausdorff metric  of $F^n(\{x\})=\{f^n(x)\}$ for any $x\in A$.
However,  it always holds that $\mathcal{B}^*_p(A)=A$.

%
%
\subsection{Chaos game} Now, we focus our study to the chaos game
of quasi-attractors of the IFS generated by $\mathscr{F}$ on $X$.
In particular, this covers the cases of pointwise attractors,
strict attractors, compact semi-attractors and minimal IFSs on a
compact topological space. First, we will give a rigourously
definition of the chaos game. \vspace{0.2cm}

Following \cite{BV11}, we consider any probability $\mathbb{P}$ on
$\Omega$ with the following property:  there exists $0<p\leq 1/k$
so that $\omega_n$ is selected randomly from $\{1,\dots,k\}$ in
such a way that the probability of $\omega_n=i$ is greater than or
equal to $p$, regardless of the preceding outcomes, for all $i\in
\{1,\dots,k\}$ and  $n\in \mathbb{N}$. More formally, in terms of
the conditional probability,
\begin{equation}
\label{conditional assumption} \mathbb{P}( \omega_n=i \ | \
\omega_{n-1},\ldots,\omega_1 ) \geq p. \end{equation} Bernoulli
measures on $\Omega$ are typical examples of these kinds of
probabilities. \vspace{0.2cm}

\begin{defi}
Let $A$ be a quasi-attractor  of the IFS generated by
$\mathscr{F}$. We say that $A$ is renderable by
\begin{enumerate}[topsep=0.2cm,itemsep=0.2cm]
\item the \emph{probabilistic chaos game} if for any $x\in
\mathcal{B}^*_{p}(A)$ there is  $\Omega(x) \subset \Omega$ with
$\mathbb{P}(\Omega(x))=1$ such that
$$
     A \subset \overline{O^+_\omega(x)}
     \quad \text{for all
     $\omega \in \Omega(x)$;}
$$
\item
the \emph{deterministic chaos game}
if there is $\Omega_0 \subset \Omega$ with
$\mathbb{P}(\Omega_0)=1$ such that
$$
     A \subset \overline{O^+_\omega(x)}
     \quad \text{for all
      $\omega \in \Omega_0$  and
     $x\in \mathcal{B}^*_p(A)$.}
$$
\end{enumerate}
\vspace{0.2cm} If the IFS is forward minimal (consequently $A$ is
the whole space and $\mathcal{B}^*_p(A)=A$) we simply say that the
IFS is renderable by the probabilistic/deterministic chaos game.
\end{defi}
\vspace{0.3cm}

The sequences in $\Omega$ which have a dense orbit under the shift
map $\sigma:\Omega \to \Omega$ are called \emph{disjunctive}. That
is, the sequence $ \omega=\omega_1\omega_2\dots \in \Omega$ which
contains all the finite words $\alpha=\alpha_1\dots\alpha_n \in
\{1,\dots,k\}^n$ of length $n$, for all $n\geq 1$. Notice that the
set consisting of all disjunctive sequences has
$\mathbb{P}$-probability one and its complement is a
$\sigma$-porous set with respect to the Baire metric in
$\Omega$~\cite{BL14}.
The following result shows that the existence of a sequence
$\omega$ such that every point in the basin of attraction has
dense $\omega$-fiberwise orbit on the self-similar set is enough
to guarantee that for any disjunctive sequence we can also draw
the quasi-attractor. This brings to light that actually the
deterministic chaos game does not depend on the probability
$\mathbb{P}$. This fact deradomizes the algorithm of the chaos
game since disjunctive sequences in $\Omega$
 are a priori well determined
sequences. For this reason, the algorithm is called the
deterministic chaos game (or disjunctive chaos game).

\begin{mainthm}
\label{thmA-chaos-game} Consider the IFS generated by
$\mathscr{F}$ and let $A$ be a compact set of $X$ and $x\in
\mathcal{B}^*_p(A)$. Then,
\begin{enumerate}[topsep=0.1cm, itemsep=0.2cm]
\item \label{item1-thmA-chaos-game} $A$ and $\overline{\Gamma(x)}$ are forward invariant
compact sets of $X$ and $A\subset \overline{\Gamma(x)}$. In
particular, for every $n\in\mathbb{N}$ and  $\omega\in \Omega$,
    $$\overline{\{\displaystyle f^m_\omega(x): m\geq n\}} \ \
    \text{is a compact set;}$$
\item \label{item2-thmA-chaos-game} $A\subset \overline{O^+_\omega(x)}$ if and only if
\begin{equation*}
    \lim_{n\to\infty} \overline{\displaystyle\{f^m_\omega(x): m\geq n\}} = A \
    \ \text{in the Vietoris topology};
\end{equation*}
\item \label{item3-thmA-chaos-game}
if $A$ is a quasi-attractor, the following are equivalent:
\begin{enumerate}[topsep=0.1cm, itemsep=0.1cm]
\item $A$ renderable by the deterministic chaos game;
\item there is $\omega \in \Omega$ such
that
$A\subset \overline{O^+_\omega(z)}$  \text{for all $z\in
\mathcal{B}^*_p(A)$};
\item $A \subset \overline{O^+_\omega(z)}$
     for all $z\in \mathcal{B}^*_p(A)$ and disjunctive sequences
     $\omega\in \Omega$.
\end{enumerate}
\end{enumerate}
\end{mainthm}

\subsubsection{Probabilistic chaos game}
Initially, the method was developed for contracting
IFSs~\cite{B88}. Later, it was generalized to attractors of  IFSs
of continuous functions on proper metric spaces~\cite{BV11}. For
minimal IFSs in the case of independent identically distributed
random product of continuous maps of a compact metric space it is
follows from the Breiman's law of large numbers~\cite{Bre60}.
Recently in~\cite{BLR15}, Barnsley, Le{\'s}niak and Rypka proved
the probabilistic chaos game for pointwise attractors of
continuous IFSs on a first-countable Hausdorff topological space
(in fact they only need to assume that the attractor is
first-countable). Moreover,
their proof of the probabilistic chaos game also works
for the general case of quasi-attractors with minor modifications
(see Appendix~\ref{s:appendix}).

%

\vspace{-0.2cm}

\subsubsection{Deterministic chaos game}
In the case of attractors  of contractive IFSs a very simple
justification of the deterministic chaos game can be given along
the lines in~\cite[proof of Thm~5.1.3]{E98}. In~\cite{BFS13}, the
deterministic algorithm was also proved for attractors of weakly
hyperbolic IFSs, i.e., for point-fibred attractors (see the below definition),
which are an extension of the previous attractors of
contractive IFSs. Later, in~\cite{BL14} the deterministic chaos
was obtained for a more general class of attractors, the so-called
strongly-fibred.

An attractor $A$ is said to be \emph{strongly-fibred} if for every
open set $U\subset X$ such that $U\cap A\not=\emptyset$, there
exists $\omega\in \Omega$ so that
$$
    A_\omega\eqdef \bigcap_{n=1}^\infty f_{\omega_1}\circ \cdots
    \circ f_{\omega_n}(A)  \subset U.
$$
Similarly, $A$ is said to be \emph{point-fibred} if $A_\omega$ is
a singleton for all $\omega\in\Omega$. Same definitions will lead
to quasi-attractors. We are going to introduce a similar category
that we will call \emph{well-fibred} following the proposal
Kieninger's classification of IFS attractors~\cite[p.~97]{K02},
\cite{BLW14}.

\begin{defi}
We say that a quasi-attractor $A$ of the IFS is \emph{well-fibred}
if for every compact set $K$ in $A$ so that $K \neq A$ and for any
open cover $\mathcal{U}$ of $A$, there exist $g \in \Gamma$ and
$U\in\mathcal{U}$ such that $g(K)\subset U$. Equivalently, if
there are $\omega\in \Omega$ and $U\in\mathcal{U}$  so that
$$
K_\omega\eqdef \bigcap_{n=1}^\infty f_{\omega_1}\circ \cdots
    \circ f_{\omega_n}(K)  \subset U.
$$
\end{defi}

It is not difficult to see that, in the metric space case, a
quasi-attractor $A$ is well-fibred if and only if for every
compact set $K$ in $A$ so that $K \neq A$, there is a sequence
$(g_n)_n \subset \Gamma$ such that the diameter $\diam g_n(K)$
converges to zero as $n \to \infty$. On the other hand, it is easy
to show that strongly-fibred implies well-fibred. In fact, we will
prove that if $f_i(A)$ is not equal to $A$ for some generator
$f_i$ then both notions, strongly-fibred and well-fibred,  are
equivalent. After this observation, we can say that the following
result generalizes~\cite{BL14}.

\begin{mainthm}
\label{thmE} Every well-fibred quasi-attractor $A$ of an IFS of
continuous maps of a Hausdorff topological space is renderable by
 the deterministic chaos game. Moreover, if $A$ is either,
strongly-fibred or the generators of the IFS restricted to $A$ are
homeomorphisms, then
$$
         \Omega \times A =  \overline{\{\Phi^n(\omega,x):
         n\in\mathbb{N}\}} \quad \text{for all disjunctive $\omega\in \Omega$ and $x\in A$.}
$$
\end{mainthm}

As a consequence, we will prove that every forward and backward
minimal IFS of homeomorphisms $\mathscr{F}$ of a metric space so
that the associated semigroup has a map with exactly two fixed
points, one attracting and one repelling, is renderable by the
deterministic chaos game
(see~Corollary~\ref{cor-forward-backward}). Backward minimality
here means that the IFS generated by $\mathscr{F}^{-1}=\{f^{-1}:
f\in \mathscr{F}\}$ is forward minimal.

New examples of attractors renderable by the deterministic chaos
game which are not necessarily well-fibred were given
in~\cite{BFS14,L14}. Namely, in~\cite{BFS14} the deterministic
chaos was proved for any forward and backward minimal IFS of
homeomorphisms of the circle and for every IFS of a compact metric
space that contains a minimal map. In~\cite{L14} the deterministic
algorithm was shown to work also for attractors of IFSs comprising
maps which do not increase distances. In fact, basically with the
same proof (see Appendix~\ref{s:appendix}), the result of
Le{\'s}niak also holds for quasi-attractors of
\emph{non-expansive} IFSs, i.e., iterated function systems
generated by a finite family $\mathscr{F}$ of maps of a metric
space $X$ so that
$$
    d(f(x),f(y))\leq d(x,y) \quad \text{for all $f\in
    \mathscr{F}$}.
$$
This class of systems include  equicontinuous IFSs
(see~\cite[Lem.~3.2]{L15} and \cite[Prop.~8]{MM15}) and  weakly
hyperbolic IFSs (see \cite[Thm.~1]{BI11} and
\cite[Cor.~6.4]{BKNNS15}). However, a priori, there are no
relations between quasi-attractors of non-expansive IFSs and
strongly-fibred or well-fibred attractors.

 In  brief, it is known that the deterministic chaos algorithm  holds
 in the following cases:
\begin{enumerate}
\item well-fibred quasi-attractors of IFSs on Hausdorff topological spaces,
\item quasi-attractors of non-expansive
IFSs on metric spaces,
\item forward and backward minimal IFSs of homeomorphisms of
the circle,
\item IFSs on a compact metric space having a minimal
map.
\end{enumerate}
The following theorem includes another different class of systems
to this list: the quasi-attractors of symmetric IFSs. We say that
an IFS generated by a  family of homeomorphisms $\mathscr{F}$ of
$X$ is \emph{symmetric} if for each $f\in \mathscr{F}$ it holds
that $f^{-1}\in\mathscr{F}$.

\begin{mainthm}
\label{thmF} Every quasi-attractor of a symmetric IFS on a
Hausdorff topological space is renderable by the deterministic
chaos game.
\end{mainthm}

We will give  examples of symmetric non-minimal IFSs with a
quasi-attractor which is not an attractor (Remark~\ref{rem-quasi})
and attractors of symmetric IFSs which are not include in the
previous list (Exemple~\ref{example-final}). Moreover, we will
prove that the phase space of a forward minimal symmetric IFS on a
connected space is a strict atractor
(Proposition~\ref{prop-quasi-symmetric}).

\subsubsection{Necessary condition to get the deterministic chaos game}
The next result goes in the direction to provide necessarily
conditions to yield the deterministic chaos game. First we need to
introduce the notion of backward minimality. A set $A$ of $X$ is
said to be \emph{backward invariant} for the IFS if $\emptyset\not
= f^{-1}(A) \subset A $ for all $f\in \Gamma$, where $f^{-1}(A)$
denotes the preimage of $A$ by the continuous map $f$. Hence, we
say that the IFS is \emph{backward minimal} if the unique backward
invariant non-empty closed set is the whole space.

\begin{mainthm}
\label{thmB} Every  forward minimal IFS generated by continuous
maps of a compact Hausdorff topological space which is renderable
by the deterministic chaos game must be also backward minimal.
\end{mainthm}

As an application of the above result we can complete the main
result in~\cite{BFS14} obtaining the following corollary:

\begin{maincor}
\label{maincor-circ} Let $f_1,\dots,f_k$ be circle homeomorphisms.
Then the following statements are equivalent:
\begin{enumerate}[topsep=0.1cm, itemsep=0.1cm]
\item \label{item-main-cor-1} the IFS generated by $f_1,\dots,f_k$ is renderable by the deterministic chaos
game;
\item \label{item-main-cor-2} there exists $\omega\in \Omega$
such that $\overline{O^+_\omega(x)}=S^1$ for all $x\in S^1$;
\item \label{item-main-cor-3} the IFS generated by $f_1,\dots,f_k$ is forward and backward minimal.
\end{enumerate}
\end{maincor}

This result allows us to construct a contra-example of the
deterministic chaos game for general IFSs. More specifically, any
forward minimal but not backward minimal IFS of homeomorphisms of
the circle is not renderable by the deterministic chaos game.
Observe that for ordinary dynamical systems on the circle, the
minimality of a map $T$ is equivalent to that of $T^{-1}$. However
this fact does not hold for IFSs with more than one generator:

\begin{maincor}
\label{maincor-contraexemplo} There exists an IFS of
homeomorphisms of the circle that is forward minimal but not
backward minimal. Moreover,  $S^1$  is a strict attractor of this
IFS which, consequently, is not renderable by  the deterministic
chaos game.
\end{maincor}

We want to indicate that, as we will see, most of the minimal IFSs
of homeomorphisms of the circle have $S^1$ as an strict attractor.
Namely, we will prove that $S^1$ is an strict attractor of a
minimal IFS of homeomorphisms of the circle if there is no common
invariant measure for the generators
(Proposition~\ref{prop:Malicet}).

\noindent {\bf Organization of the paper:} In \S\ref{sec2} we
study the basin of attraction  of pointwise/strict attractors and
we prove Theorem~\ref{thmA-attractor} and the two first conclusions of
Theorem~\ref{thmA-chaos-game}. We complete the proof of
Theorem~\ref{thmA-chaos-game} in \S\ref{sec41} where we study the
deterministic chaos game. In \S\ref{sec42} we prove
Theorem~\ref{thmB} and in \S\ref{sec43} we study the deterministic
chaos game on the circle proving Corollaries~\ref{maincor-circ}
and~\ref{maincor-contraexemplo}. The proof of Theorems~\ref{thmE}
and~\ref{thmF} are developed in \S\ref{sec44} where we study
sufficient conditions for the deterministic chaos game. Finally,
for completeness of the paper, we include an appendix where we
extend the main results of~\cite{BLR15} and \cite{L14} for the
general case of quasi-attractors.


\vspace{0.5cm}

\noindent {\bf Standing notation:} In the sequel, $X$ denotes a
Hausdorff topological space. We assume that we work with an IFS of
continuous maps $f_1,\dots,f_k$ on $X$ and we hold the above
notations introduced in this section.

\section{On the basin of attraction}
\label{sec2} We will study the basin of attraction of
quasi-attractors. This allows us to prove
Theorem~\ref{thmA-attractor} and Conclusions~\ref{item1-thmA-chaos-game}
and \ref{item2-thmA-chaos-game} of Theorem~\ref{thmA-chaos-game}.

\subsection{Topological preliminaries:}
We start giving a basic topological lemma:

\begin{lem}
\label{lem-topologia-X} Let $A$ and $B$ be two compact sets in
$X$.
\begin{enumerate}
\item If $A\cap B=\emptyset$ then
there exist disjoint open neighborhoods of $A$ and $B$;
\item If $\{U_1,\dots,U_s\}$ is a  finite open cover of $A$ then
there exist compact sets $A_1, \dots, A_s$ in $X$ so that
$$
A=A_1\cup\dots\cup A_s \ \ \text{and} \ \ A_i\subset U_i \ \
\text{for $i=1,\dots,s$}.
$$
\end{enumerate}
\end{lem}
\begin{proof}
%
%
The first conclusion is a well known equivalent definition of a Hausdorff
topological space (see \cite[Lemma 26.4 and Exercice 26.5]{Mu00}).
Hence, we only need to prove the second conclusion. First of all, notice
that it suffices to prove the result for an open cover of $A$ with
two sets. So, let $\{U_1,U_2\}$ be an open cover of $A$. Since
$X$ is Hausdorff, $A$ is a closed subset of $X$. Let us consider
compact subsets $K_1 = A \setminus U_2 \subset U_1$ and $K_2 = A
\setminus U_1 \subset U_2$. If $A \subset K_1\cup K_2$ then we set
$A_1 = K_1$, $A_2 = K_2$ and it is done. Otherwise, $A\setminus
(K_1\cup K_2)$ is a nonempty subset of $A$ and it is easy to see
that $K_1 \cap K_2 = \emptyset$. Since $X$ is Hausdorff and $K_1$
and $K_2$ are compact disjoint subsets of $X$, by the first conclusion,
there are disjoint open subsets $V_i$ of $X$ so that $K_i \subset
V_i$, for $i = 1,2$. We may assume that $V_i \subset U_i$. Now let
us take $A_1 = A\setminus V_2 \subset U_1$ and $A_2 = A\setminus
V_1 \subset U_2$. Then $A_1$ and $A_2$ are compact subsets of $X$
and $A = A_1 \cup A_2$ which concludes the proof. 
\end{proof}

Let $A$ and  $A_n$, $n\geq 1$, be compact subsets of $X$.
Following~\cite{B93}, we define the \emph{upper Kuratowski limit}
of $(A_n)_{n}$ as the set
$$
  \Ls A_n \eqdef \bigcap_{m\geq 1} \overline{\bigcup_{n\geq m}
  A_n}.
$$
Observe that $\Ls A_n$ is a closed set and $\Ls A_n \subset B$ that
provide $A_n \subset B$, for $n$ sufficiently large, and $B$ is a
closed set.
On the other hand, we recall that the Vietoris topology in
$\mathscr{H}(X)$ is generated by the basic sets of the form
$$
  O\langle U_1,\dots,U_m \rangle =\{K \in \mathscr{H}(X): K \subset
   \bigcup_{i=1}^m U_i, \  K \cap U_i \not= \emptyset \ \text{for $k=1,\dots,m$}
  \}
$$
where $U_1,\dots U_m$ are open sets in $X$ and $m\in\mathbb{N}$.
Hence, if $A_n\to A$ in the Vietoris topology then $A_n \in
O\langle U\rangle$ for any $n$ large enough and any open set $U$
in $X$ such that $A \subset U$. In particular $A_n\subset U$ for
all $n$ sufficiently large. Moreover we have the following:

\begin{lem}
\label{lem-conv-Vietoris}  $A_n \to A$ in the Vietoris topology if
and only if
for any pair of open sets $U$ and $V$ such that $A\subset U$ and
$A\cap V \not=\emptyset$, there is $n_0\in\mathbb{N}$ so that
$$
\bigcup_{n\geq n_0} A_n \subset U \quad \text{and} \quad V\cap A_n
\not=\emptyset \ \ \text{for all $n\geq n_0$}.
$$
In particular,
\begin{equation} \label{eq:conteudo}
 A = \Ls A_n \eqdef \bigcap_{m\geq 1} \overline{\bigcup_{n\geq m} A_n}. 
\end{equation}
\end{lem}
\begin{proof}
Assume that $A_n\to A$ in the Vietoris topology. Let $U$ be any
open set such that $A\subset U$. By applying the above observation,
there is $n_0\in\mathbb{N}$ such that $A_n \subset U$ for all
$n\geq n_0$. Now we will see that for any open set $V$ with $A\cap
V \not=\emptyset$, it holds that $A_n \cap V\not =\emptyset$ for
all $n$ sufficiently large. By the compactness of $A$, we extract
open sets $U_1,\dots, U_s$ in $X$ such that
$$
\text{$A\cap U_i \not =\emptyset$ \quad and \quad  $A \subset
V\cup U_1\cup \dots \cup U_s$.}
$$
Hence $O\langle V, U_1,\dots,U_s\rangle$ is
an open neighborhood of $A$ in $\mathscr{H}(X)$. Since $A_n$
converges to $A$  then $A_n\in O\langle V, U_1,\dots,U_s\rangle$
for all $n$ large enough and in particular $A_n \cap V\not
=\emptyset$ for all $n$ large. \vspace{0.1cm}

We will prove the converse. Let $O\langle U_1,\dots,U_m \rangle$
be an basic open neighborhood of $A$. Thus, $U_1,\dots, U_m$ are
open sets in $X$ and
$$
A\subset U_1 \cup \dots \cup U_m \eqdef U \quad \text{and} \quad
A\cap U_i \not= \emptyset \ \ \text{for all $i=1,\dots,m$.}
$$
By assumption, there is $n_0$ such that $A_n \subset U$ for all
$n\geq n_0$. Moreover, since $A\cap U_i \not=\emptyset$, also we
get  $n_i$ such that $A_n \cap U_i \not=\emptyset$ for all $n\geq
n_i$ and $i=1,\dots,m$. Therefore
$$
\text{$A_n \in O\langle U_1,\dots,U_m \rangle$ \ \ for all $n\geq
N=\max\{n_i: i=0,\dots,m\}$.}
$$
This implies that $A_n \to A$ in
the Vietoris topology. \vspace{0.1cm}

Finally we will prove~\eqref{eq:conteudo}. We have that $A\subset
\Ls A_n$ since for every open neighborhood $V$ of any point in $A$
there is $n_0\in\mathbb{N}$ such that
$$
\text{$A_n\cap V \not=\emptyset$ \ \ for all $n\geq n_0$.}
$$
Reverse
content is equivalent to prove that for every compact set $K$ such
that $K \cap A=\emptyset$, there exists $n_0\in \mathbb{N}$ so
that $A_n \cap K=\emptyset$ for all $n\geq n_0$. But this is a
consequence again of Lemma~\ref{lem-topologia-X}. Indeed, since
$K$ and $A$ are compact sets, we can find disjoint open sets $U$
and $V$ such that $A \subset U$ and $K\subset V$. By the above
characterization of the Vietoris convergence, there is
$n_0\in\mathbb{N}$ such that $A_n \subset U$ for all $n\geq n_0$.
In particular $A_n\cap K=\emptyset$ for all $n\geq n_0$. \qed
\end{proof}


\subsection{Proof of Theorem~\ref{thmA-attractor}
and Conclusions~\ref{item1-thmA-chaos-game}-\ref{item2-thmA-chaos-game}
of Theorem~\ref{thmA-chaos-game}}
 We start proving Conclusion~\ref{item1-thmA-chaos-game} of
Theorem~\ref{thmA-chaos-game}.

\begin{pro}
\label{lem-compact-orbit} Let $A$ be a compact subset of $X$. If
$x \in \mathcal{B}^*_p(A)$ then both, $A$ and
$\overline{\Gamma(x)} $, are forward invariant compact sets such
that
$$
     A= \Ls F^n(\{x\})\eqdef \bigcap_{m\geq 1} \overline{\bigcup_{n\geq m} F^n(\{x\})}
     \quad \text{and} \quad \overline{\Gamma(x)}=\bigcup_{n\geq 1} F^n(\{x\}) \cup
     A.
$$
In particular,
$$
    \overline{\{f^m_\omega(x): m\geq n\}} \ \text{is a compact set
    for all $n\in\mathbb{N}$ and  $\omega\in \Omega$.}
$$
\end{pro}
\vspace{0.01cm}
\begin{proof}
Set $K\eqdef \overline{\Gamma(x)}$. Since $x\in
\mathcal{B}_p^*(A)$, by definition it follows the above
characterization of $A$, and consequently of $K$. \vspace{0.01cm}
\vspace{0.1cm}

Now we will show that $K$ is compact. Let $\{U_\alpha: \alpha \in
I\}$ be an open cover of $K$. Since $A\subset K$, by the
compactness of $A$ there exists a finite subset $J_1$ of $I$ such
that
$$
    A \subset \bigcup_{\alpha \in J_1} U_\alpha \eqdef U.
$$
Again, by the above characterization of the set $A$ and since
\begin{equation*}
\label{eq:Bm} \overline{\bigcup_{n\geq m} F^n(\{x\})} \quad
\text{for $m\geq 1$}
\end{equation*}
is a nested sequence then there is $n_0\in \mathbb{N}$ such that
the union of $F^n(\{x\})$ for $n\geq n_0$ is contained in $U$. On
the other hand, the set $F(\{x\})\cup \dots \cup F^{n_0-1}(\{x\})$
is a finite union of compact sets and thus, it is compact. Hence,
there is a finite subset $J_2$ of $I$ such that
$$
   F(\{x\})\cup
\dots \cup F^{n_0-1}(\{x\})  \subset \bigcup_{\alpha \in J_2}
U_\alpha.
$$
Put together all and setting $J=J_1\cup J_2$ we get that
$$
   K=\overline{\Gamma(x)}
   = A\cup F(\{x\})\cup
\dots \cup F^{n_0-1}(\{x\}) \cup \bigcup_{n\geq n_0} F^n(\{x\})
\subset \bigcup_{\alpha \in J} U_\alpha
$$
concluding that $K$ is compact. \vspace{0.1cm}

Moreover clearly $F(K)\subset K$. Thus we have obtained that $K$
is a compact Hausdorff topological space so that $F(K)\subset K$
and $A\subset K$. Hence, we can restrict the map $F$ to the set of
non-empty compact subsets of $K$.

According
to~\cite[Prop.~1.5.3~(iv)]{K02}, see also~\cite{BL15}, the
Hutchinson operator $F:\mathscr{K}(K)\to \mathscr{K}(K)$ is
continuous and from the above characterization of the set $A$, it
is easy to conclude that $A$ is also a forward invariant compact
set. This completes the proof of the proposition. \qed
\end{proof}

Now, we characterize the quasi-attractors
(Conclusions~\ref{item1-thmA-attractor}-\ref{item1-thmA-attractor3}
of Theorem~\ref{thmA-attractor}).
\begin{pro}
\label{pro:self-similar} Let $A$ be a compact subset of $X$. Then
\begin{enumerate}[topsep=0.1cm,itemsep=0.1cm]
\item $\mathcal{B}_p(A) \subset \mathcal{B}_p^*(A)$;
\item $A$  is a quasi-attractor if and only if $A\subset
\mathcal{B}^*_p(A)$. Moreover, in this case,
\begin{gather*}
  \Ls F^n(K)=A \quad \text{for all non-empty compact set $K\subset
  A$ \  and } \\
\text{$\overline{\Gamma(x)}\subset \mathcal{B}^*_p(A)$ \ \ for all
$x\in\mathcal{B}^*_p(A)$;}
\end{gather*}
\item if $A$ is a pointwise attractor, it is a
quasi-attractor and $\mathcal{B}_p(A) = \mathcal{B}_p^*(A)$.
\end{enumerate}
\end{pro}

\begin{proof}
The first conclusion follows from the
characterization~\eqref{eq:conteudo} of the limit of $F^n(\{x\})$
in the Vietoris topology given in Lemma~\ref{lem-conv-Vietoris}.
\vspace{0.1cm}

Assume that $A$ is a quasi-attractor and let $x\in A$. We want to
prove that $A=\mathrm{Ls}\, F^n(\{x\})$. Since $F^n(\{x\})\subset
A$ for all $n\geq 1$ and $A$ is a closed set then $\Ls F^n(\{x\})
\subset A$. As in Proposition~\ref{lem-compact-orbit}, $\Ls
F^{n}(\{x\})$ is a forward invariant closed set and thus, by the
minimality of $A$,  $\mathrm{Ls}\, F^n(\{x\})=A$ and  $x\in
\mathcal{B}^*_p(A)$. Moreover, the same argument also proves that
$\Ls F^n(K)=A$ for all non-empty compact set $K\subset A$. In
fact, considering that Proposition~\ref{lem-compact-orbit} implies
that
$$
\overline{\Gamma(y)}= \Gamma(y)\cup A \quad \text{for all $y\in
\mathcal{B}^*_p(A)$},
$$
for all $z\in \overline{\Gamma(y)}$ it holds that if $z\in A$ then
$\Ls F^n(\{z\})=A$ and if $z\in \Gamma(y)$ then $\Ls F^n(\{z\})
\subset \Ls F^n(\{y\})=A$ and from the above arguments $\Ls
F^n(\{z\})=A$. Therefore the closure of $\Gamma(y)$ is contained
in $\subset \mathcal{B}_p^*(A)$ for all $y\in \mathcal{B}^*_p(A)$.

%

\vspace{0.1cm}
 Suppose now that $A\subset \mathcal{B}^*_p(A)$. Hence
$A=\mathrm{Ls}\, F^n(\{x\})$ for all $x\in A$. In particular, $A$
is a forward minimal set and by
Proposition~\ref{lem-compact-orbit} also is a forward invariant
set. Thus $A$ is a forward invariant minimal set, that is, a
quasi-attractor.

\vspace{0.1cm} Finally we will proof the last conclusion. By the
first conclusion, it suffices to show that if $A$ is an strict
attractor then $\mathcal{B}^*_p(A)\subset \mathcal{B}_p(A)$. To
accomplish this, let us consider $x\in \mathcal{B}^*_p(A)$ and
open sets $U$, $V$ such that $A\subset U$ and $A\cap
V\not=\emptyset$. Being $A$ a pointwise attractor there is an open
neighborhood $W$ of $A$ so that $F(\{z\})\to A$ in the Vietoris
topology for all $z\in W$. Without loss of generality, we can
assume that $U\subset W$. Since $A=\mathrm{Ls} F^n(\{x\})$ then
there exists $n_1\in \mathbb{N}$ such that $F^n(\{x\})\subset U$
 for all $n\geq n_1$.
Then, for every $z\in F^{n_1}(\{x\})$ we have that $F^n(\{z\})$
converges to $A$ in the Vietoris topology and thus, by
Lemma~\ref{lem-conv-Vietoris}, there is $n_2=n_2(z)\in \mathbb{N}$
so that
$$ \text{$F^{n}(\{z\})\cap V \not = \emptyset$ \ \ for all $n\geq
n_2$.}
$$
Taking $n_0=\max\{n_2(z): z\in F^{n_1}(\{x\})\}$ we get that
$F^n(\{x\})\cap V\not=\emptyset$ for all $n\geq n_0$ and
therefore, Lemma~\ref{lem-conv-Vietoris} implies that $x\in
\mathcal{B}_p(A)$  \qed
\end{proof}

\vspace{0.2cm} We complete the proof of
Theorem~\ref{thmA-attractor} by proving
Conclusion~\ref{item2-thmA-attractor}.

\begin{pro}
\label{prop-attractor} 
If $A$ is a strict attractor, then
\begin{enumerate}
\item $F^n(K)\to A$ in the Vietoris topology for all compact sets $K \subset
\mathcal{B}(A)$;
\item $A$ is a pointwise attractor and $\mathcal{B}(A)=\mathcal{B}_p(A)$.
\end{enumerate}

\end{pro}
\begin{proof}
The first conclusion is a consequence of Lemma~\ref{lem-topologia-X}.
Indeed, given any compact set $K$ in $\mathcal{B}(A)$, by
compactness we can find open neighborhoods $U_1,\dots, U_s$ of $A$
such that $K\subset U_1\cup\dots\cup U_s$ and  $F^n(S)\to A$ for
any compact set $S$ in $U_i$, for all $i=1,\dots,s$. By
Lemma~\ref{lem-topologia-X} there are compact sets $K_i \subset
U_i$ for $i=1,\dots,s$ such that $K=K_1\cup \dots \cup K_s$. Then,
$$F^n(K)=F^n(K_1)\cup \dots \cup F^n(K_s)$$ and thus $F^n(K)$
converges to $A$ in the Vietoris topology.

\vspace{0.1cm} We will prove the second conclusion. By means of the
first conclusion, $\mathcal{B}(A)\subset \mathcal{B}_p(A)$. Thus, since
$\mathcal{B}(A)$ is an open set containing $A$ we get that $A$ is
a pointwise attractor. To conclude, we will show that
$\mathcal{B}_p(A) \subset \mathcal{B}(A)$. Given $x\in
\mathcal{B}_p(A)$ we want to prove that $x$ belongs to
$\mathcal{B}(A)$.
\renewcommand{\theclaim}{\arabic{section}.\arabic{thm}.\arabic{claim}}
\begin{claim}
\label{claim1} If there exists a neighborhood $V$ of $x$ such that
$F^n(K)\to A$ in the Vietoris topology for all non-empty compact
sets $K \subset V$ then $x\in \mathcal{B}(A)$.
\end{claim}
\begin{proof}
Since $A$ is an attractor there exists a neighborhood $U_0$ of $A$
such that $F^n(S) \to A$ for all compact sets $S$ in $U_0$. Take
$U= U_0 \cup V$. Clearly, $U$ is a neighborhood of $A$ and $x\in
U$. On the other hand, by Lemma~\ref{lem-topologia-X}, any compact
set $K$ in $U$ can be written as the union of two compact sets
$K_0$ and $K_1$ contained in $U_0$ and $V$ respectively. Now,
since $ F^n(K)=F^n(K_0)\cup F^n(K_1) $ it follows that $F^n(K)$
converges to $A$ for all non-empty compact sets $K$ in the
neighborhood $U$ of $A$. This implies that $x\in \mathcal{B}(A)$.
\qed
\end{proof}
Now, we will get a neighborhood $V$ of $x$ in the assumptions of
the above claim. Since $\mathcal{B}_p(A)$ is an open neighborhood
of $A$ and $x\in \mathcal{B}(A)$, we get $m \in \mathbb{N}$ such
that $F^m(\{x\})\subset \mathcal{B}(A)$. Equivalently,
$$
f_{\omega_m} \circ \cdots \circ f_{\omega_1}(x) \in \mathcal{B}(A)
\ \ \text{for all $\omega_i \in \{1,\dots,k\}$ for $i=1,\dots,m$}.
$$
By the continuity of the generators $f_1,\dots,f_k$ of the IFS, we
get an open set $V$ such that $x\in V$ and
$$
\text{$f_{\omega_m} \circ \cdots \circ f_{\omega_1}(V) \subset
\mathcal{B}(A)$ \ \ for all $\omega_i \in \{1,\dots,k\}$  for
$i=1,\dots,m$.}
$$
In particular, for every compact set $K$ in $V$
it holds that $F^m(K) \subset \mathcal{B}(A)$ and thus, by the
first conclusion, $F^{n}(K)$ converges to $A$. Finally
Claim~\ref{claim1} implies that $x \in \mathcal{B}(A)$ as we
wanted to show. This completes the proof. \qed
\end{proof}

To end this section we will prove Conclusion~\ref{item2-thmA-chaos-game}
of Theorem~\ref{thmA-chaos-game}.

\begin{pro}
Consider  $\omega \in \Omega$ and $x\in \mathcal{B}^*_p(A)$. Then
the following statements are equivalent:
\begin{enumerate}[topsep=0.2cm,itemsep=0.2cm]
\item \label{1} $A\subset \overline{O^+_\omega(x)};$ 
\item \label{2}
$\displaystyle \lim_{n\to\infty} \overline{\{f^m_\omega(x): m\geq
n\}} = A \ \ \text{in the Vietoris topology;} $
\item \label{3}
$\displaystyle
         A=\bigcap_{n\geq 1} \overline{\{f^m_\omega(x): m\geq
         n\}}.
$
\end{enumerate}
\end{pro}
\begin{proof}
According to Proposition~\ref{lem-compact-orbit},
$A_n=\overline{\{f_\omega^m(x): m\geq n\}}$ is a compact set for
all $n\in\mathbb{N}$. Moreover, $A_{n+1} \subset A_n$ and hence,
by Lemma~\ref{lem-conv-Vietoris}, if  $A_n\to A$ in the Vietoris
topology,
$$
     A= \bigcap_{n\geq 1} A_n \subset A_1 =
     \overline{O_\omega^+(x)}.
$$
This proves~\ref{2} implies~\ref{1}.

Reciprocally, let $U$ and $V$ be open sets such that $A\subset U$
and $V\cap A \not=\emptyset$. Since $x\in \mathcal{B}^*_p(A)$ then
$A=\mathrm{Ls}\,F^n(\{x\})$ and thus there exists $n_0\in
\mathbb{N}$ such that
$$
\bigcup_{n\geq n_0} F^n(\{x\})\subset U.
$$
In particular, the  union of $A_n$ for $n\geq n_0$ is contained in
$U$. Moreover, since $A\subset \overline{O_\omega^+(x)}$ we have
$A_n\cap V\not=\emptyset$ for all $n$ large enough.
Lemma~\ref{lem-conv-Vietoris} implies that $A_n \to A$ in the
Vietoris topology completing the proof of \ref{1} implies \ref{2}.

Finally, by Lemma~\ref{lem-conv-Vietoris}, we have that \ref{2}
implies \ref{3} and easily one can see that~\ref{3}
implies~\ref{1} concluding the proof of the proposition \qed
\end{proof}

\section{Deterministic chaos game}

\subsection{Equivalence}
\label{sec41} We will conclude Theorem~\ref{thmA-chaos-game}
proving Conclusion~\ref{item3-thmA-chaos-game}.

\begin{pro}
Let $A$ be a quasi-attractor. Then the following statements are
equivalent:
\begin{enumerate}[itemsep=0.1cm]
\item \label{item-one_seq} there exists $\omega\in \Omega$ such that $A \subset
\overline{O^+_\omega(x)}$ for all $x\in \mathcal{B}^*_p(A)$;
\item \label{item-disj_seq} $A\subset \overline{O^+_\omega(x)}$ for all disjunctive sequences
$\omega\in\Omega$ and $x\in \mathcal{B}^*_p(A)$;
\item there is $\Omega_0 \subset \Omega$ with $\mathbb{P}(\Omega_0)=1$ such that
$$
     A \subset \overline{O^+_\omega(x)}
     \quad \text{for all
      $\omega \in \Omega_0$  and
     $x\in \mathcal{B}^*_p(A)$.}
$$
\end{enumerate}
\end{pro}
\begin{proof}
It suffices to show that~\ref{item-one_seq}
implies~\ref{item-disj_seq}. Let $x$ be a point in
$\mathcal{B}^*_p(A)$. According to
Proposition~\ref{lem-compact-orbit}, $$K\eqdef
\overline{\Gamma(x)} \subset \mathcal{B}^*_p(A)$$ and it is a
forward invariant compact set.

The following claim will be useful to prove the density of
disjunctive fiberwise orbits, i.e, of fiberwise orbits driving by
disjunctive sequences: \setcounter{claim}{0}
\begin{claim}
\label{claim-desisdad-disj} Let $Z$ be a forward invariant set
such that $A\subset Z$. If for any non-empty open set $I\subset X$
with $A\cap I \not =\emptyset$, there is $f_{i_s}\circ \cdots
\circ f_{i_1} \in \Gamma$ such that
\begin{equation*}
\label{eq:1} \text{for each $z\in Z$ there is $t\in \{1,\dots,s\}
$ so that $f_{i_t}\circ \cdots \circ f_{i_1}(z)\in I$}
\end{equation*}
then
$$
A \subset \overline{O^+_\omega(x)} \quad \text{for all disjunctive
sequences $\omega\in \Omega$ and $x\in Z$.}
$$
\end{claim}
\begin{proof}
Consider any open set $I$ such that $A\cap I \not = \emptyset$,
$x\in Z$ and a disjunctive sequence $\omega\in \Omega$. Using the
fact that $\omega$ is a disjunctive sequence and that $Z$ is a
forward invariant set we can choose $m\geq 1$ such that
$$
[\sigma^m(\omega)]_j = i_j \ \ \text{for $j = 1,\dots, s$  \ \ and
\ \ $z=f^{m}_\omega(x)\in Z$.}
$$
Hence, by assumption, there exists $t=t(z)$ such one has that
$f^{m+t(z)}_\omega(x) \in I$ which proves the density on $A$ of
the $\omega$-fiberwise orbit of $x$. \qed
\end{proof}

Notice that $F(K)\subset K$ and hence we can take as $Z=K$ in the
above claim. Let $I$ be an open set so that $I\cap
A\not=\emptyset$. By assumption, since $Z\subset
\mathcal{B}_p(A)$, there exists a sequence $\omega\in \Omega$ such
that for each point $z \in Z$ the $\omega$-fiberwise orbit of $z$
is dense in $A$. In particular, there is $n=n(z)\in \mathbb{N}$
such that
$$\{f^m_\omega(z): m\leq n\}\cap I \not=\emptyset.$$
 By
continuity of the generators $f_1,\dots,f_k$ of the IFS, there
exists an open neighborhood $V_z$ of $z$ such that
$$\text{$\{f^m_\omega(y): m\leq n\}\cap I\not=\emptyset$ \ \ for all $y\in
V_z$.}
$$
Then, the compactness of $Z$ implies that we can extract
open sets $V_1, \dots, V_r$ and positive integers $n_1,\dots, n_r$
such that $Z\subset V_1\cup \dots \cup V_r$ and
$$
\text{$\{f^{m}_\omega(z): m\leq n_i\}\cap I \not=\emptyset$ \ \
for all $z\in V_i$ and $i=1,\dots,r$.}
$$
Hence the assumptions of Claim~\ref{claim-desisdad-disj} hold
taking $f_{\omega_s}\circ \cdots \circ f_{\omega_1}\in \Gamma$
where $s=\max\{n_i: i=1,\dots,r\}$. Therefore, since the initial
point $x \in \mathcal{B}^*_p(A)$ belongs to $Z$, we conclude that
any disjunctive fiberwise orbit of $x$  is dense in $A$ that
completes the proof. \qed
\end{proof}

\subsection{Necessary condition}
\label{sec42}
We will prove Theorem~\ref{thmB}.

\begin{proof}[Proof of Theorem~\ref{thmB}]
Clearly if there is a minimal orbital branch, i.e.,
$\omega=\omega_1\omega_2\dots \in \Omega$ such that
$O^+_\omega(x)$ is dense for all $x$, then the IFS is forward
minimal.

We will assume that it is not backward minimal. Then, there exists
a non-empty closed set $K \subset X$ such that $\emptyset\not=
f^{-1}(K) \subset K \not= X$ for all $f\in \Gamma$. We can
consider
$$
K^-_n = \bigcap_{i=1}^n f^{-1}_{\omega_1}\circ\cdots\circ
f_{\omega_i}^{-1}(K)=f^{-1}_{\omega_1}\circ\cdots\circ
f_{\omega_n}^{-1}(K) \quad \text{and} \quad K^-_\omega=
\bigcap_{n=1}^\infty K^-_n.
$$
Hence $K^-_n$ is a nested sequence of closed sets. By assumption
of this theorem, the space $X$, where the IFS is defined, is a
compact Hausdorff topological space. As a consequence,
$K^-_\omega$ is not empty and then for every $x\in K^-_\omega$ we
have that $O^+_\omega(x) \subset K$. Since $K$ is not equal to $X$
it follows that there exists a point $x\in X$ so that the
$\omega$-fiberwise orbit of $x$ is not dense. But this is a
contradiction and we conclude the proof. \qed
\end{proof}

 As in the introduction we notified, an IFS is
forward minimal if and only if every point has dense
$\Gamma$-orbit. To complete the section we want to point out the
following straightforward similar equivalent definition of
backward minimality.

\begin{lem}
Consider an IFS of surjective continuous maps of a topological
space $X$.
Then the IFS is backward minimal if and only if $
X=\overline{\Gamma^{-1}(x)}$ for all $x\in X$ where
$$
      \Gamma^{-1}(x) \eqdef \{ y\in X: \text{there exists} \ g\in \Gamma \
\text{such that} \ g(y)=x\}.$$
\end{lem}

%

\subsection{Minimal IFSs of homeomorphisms of the circle}
\label{sec43} As consequence of Theorem~\ref{thmB}, we will obtain
that the deterministic chaos game is totally characterized for
forward minimal IFSs of homeomorphisms of the circle. Moreover,
this characterization allows us to construct attractors of IFSs
that is not renderable by the deterministic chaos game
(counterexamples). \vspace{-0.1cm}
\subsubsection{Characterization} In~\cite[Thm.~A]{BFS14} it was proved that
every forward and backward minimal IFS of preserving-orientation
homeomorphisms of the circle is renderable by the deterministic
chaos game. However, the assumption of preserving-orientation can
be removed from this statement as we explain below.

The main tool in the proof of~\cite[Thm.~A]{BFS14} was~Antonov's
Theorem~\cite{An84} (see~\cite[Thm.~2.1]{BFS14}). This theorem is
stated for preserving-orientation homeomorphisms of the circle.
Supported in this result the authors showed a key lemma
(see~\cite[Lem.~2.2]{BFS14}) to prove the above statement. In
fact, in this lemma, by means of Antonov's result, is the unique
point in the proof where the preserving-orientation assumption is
used. This lemma can be improved removing the preserving
orientation assumption by two different ways. The first is
observing that in fact, this assumption is not necessarily in the
original proof of Antonov as easily one can follow from the
argument described in~\cite[proof of Theorem~2]{GGKV14}. Another
way is to use the recently generalization of Antonov's
result~\cite[Thm.~D]{M15} instead the key lemma above mentioned.

\begin{proof}[Proof of Corollary~\ref{maincor-circ}]
From above,  every forward and backward minimal IFS of
homeomorphisms of the circle is renderable by the deterministic
chaos game. That is, \ref{item-main-cor-3}
implies~\ref{item-main-cor-1}. On the other hand,
\ref{item-main-cor-1} implies~\ref{item-main-cor-3} follows from
Theorem~\ref{thmB}. Finally, to complete the proof of the
corollary it suffices to note that according to
Theorem~\ref{thmA-chaos-game}, \ref{item-main-cor-1}
and~\ref{item-main-cor-2} are equivalent. \qed
\end{proof}

\subsubsection{Counterexample} We will prove now
Corollary~\ref{maincor-contraexemplo}. As in the introduction we
mentioned, for ordinary dynamical systems, the minimality of a map
$T$ is equivalent to that of $T^{-1}$. Nevertheless this is not
the case for dynamical systems with several maps as Kleptsyn and
Nalskii pointed at~\cite[pg.~271]{KN04}. However, they omitted to
include these examples of forward but not backward minimal IFSs.
Hence, to provide a complete proof of
Corollary~\ref{maincor-contraexemplo} we will show that indeed
such IFSs of homeomorphisms of $S^1$ can be constructed. \\

\noindent {\sf a) Forward but not backward minimal IFSs on the
circle:}
%
Consider  a group $G$ of homeomorphisms of the circle. Then, there
can occur only one of the following three
options~\cite{Navas,Gh01}:
\begin{enumerate}[itemsep=0.1cm]
\item there is a finite $G$-orbit,
\item every $G$-orbit is dense on the circle, or
\item there is a unique $G$-invariant minimal Cantor set.
\end{enumerate}
\vspace{0.1cm} By a  $G$-orbit we understand the action of $G$ at a
point $x\in S^1$. That is the set of points $G(x)=\{g(x): g\in
G\}$. If $G(x)$ has finitely many different elements then it is
called \emph{finite orbit} while if its closure is $S^1$, it say
\emph{dense orbit}. The Cantor set $K$ in the above third conclusion is
usually called \emph{exceptional minimal set}. This set is
$G$-invariant and minimal, that is,
$$
g(K)=K \ \ \text{for all $g\in G$} \quad  \text{and} \quad
K=\overline{G(x)} \ \
   \text{for all $x\in K$}.
$$
Notice that these properties are the same to say that $K$ is
minimal regarding to the inclusion of $G$-invariant closed sets.
The following proposition is stated in~\cite[Exer.~2.1.5]{Navas}.
For completeness, we include the proof.
\begin{pro}\label{transnonmin}
There exists a finitely generated group $G$ of homeomorphisms of
$S^1$ admitting an exceptional minimal set $K$ such that the
$G$-orbit of every point of $S^1\setminus K$ is dense in $S^1$.
\end{pro}
\begin{proof}
Let $f$ a homeomorphism of the circle with a minimal exceptional
set $K$ and such that there is only one class of gaps, which means
that for every gaps $I$, $J$, there exists $n$ in $\mathbb{Z}$
such that $f^n(I)=J$. For instance, the classic Denjoy map. Let
$I_0$ be a gap of $K$. Let $u:I_0\rightarrow \mathbb{R}$ be a
homeomorphism,  $\tilde{f}_1$ and $\tilde{f}_2$ be respectively
the translations $x\mapsto x+1$ and $x\mapsto x+\sqrt{2}$ on
$\mathbb{R}$, and let us set define two homeomorphisms $f_1$ and
$f_2$ of $S^1$ by $f_i=u^{-1}\tilde{f}_iu$ on $I_0$,
$f_i=\mathrm{id}$ on $S^1\setminus I_0$.
 We claim that the group $G$ generated by
 $f$, $f_1$ and $f_2$ satisfies the required properties.
 Obviously, $K$ is also the minimal exceptional of $G$ since
 $f_i|_K=\mathrm{id}$. On the other hand, the subgroup $H$
 generated by $f_1$ and $f_2$ leaves the gap
$I_0$ invariant and acts minimally on it since the group generated
by $\tilde{f}_1$ and $\tilde{f}_2$ acts minimally on $\mathbb{R}$.
Hence, let $x$ be in $S^1\setminus K$ and $I$ be an interval of
the circle. Since there is only one class of gaps one can find $m$
and $n$ in $\mathbb{Z}$ such $f^m(x)\in I_0$ and $f^n(I)\cap
I_0\not=\emptyset$. Next, by minimality of the action of $H$ on
$I_0$, one can find $h$ in $H$ such that $h(f^m(x))\in f^n(I)$.
Thus, the element $g=f^{-n}hf^m$ of $G$ sends $x$ into $I$. Since
$I$ is arbitrary by $G$, the orbit of $x$ by $G$ is
    dense. \qed
\end{proof}

We say that a subset $\mathscr{F}$ of a group $G$ is a
\emph{symmetric generating system} of $G$ if $G$ is generated by
$\mathscr{F}$ as a semigroup. Moreover, we ask that if $f\in
\mathscr{F}$ then also $f^{-1}\in \mathcal{F}$. Hence, we can see
the action of the group as a symmetric IFS generated by
$\mathscr{F}$.

\begin{rem}
\label{rem-quasi} Let $\mathscr{F}$ be a finite symmetric
generating system of the group $G$ given in
Proposition~\ref{transnonmin}. Hence, from the above observation,
it follows that the exceptional minimal set $K$ is the unique
quasi-attractor of the symmetric IFS generated by $\mathscr{F}$
and it holds that $\mathcal{B}^*_p(K)=K$. This provides an example
of a quasi-attractor of a non-minimal IFS which cannot be a
pointwise attractor.
\end{rem}
 We will use the following:
\begin{lem}\label{3invariant}
Let $G$ and $K$ be as in Proposition \ref{transnonmin}. Then the
closed subsets of $S^1$ which are invariant by $G$ are $\emptyset$,
$K$ and $S^1$.
\end{lem}
\begin{proof}
Let $B$ be a closed subset of $S^1$ invariant by $G$. If $B\not=
\emptyset$, then $K\subset B$ by minimality of $K$, and if
$B\not=K$, it means that $B$ contains a point $x$ in $S^1\setminus
K$, and by invariance, $B$ contains the orbit of $x$ by $G$ which
is dense, hence $B=S^1$.
\end{proof}

 On the other hand,
any two Cantor set are homemorphic. In fact, if $K_I$ and $K_J$
are two Cantor sets in an interval $I$ and $J$ respectively, there
exists a homeomorphisms $g:I\to J$ so that $g(K_I)=K_J$ (see for
instance \cite{Bo75}). Hence given any Cantor set $K$ in $S^1$ one
can find a homeomorphisms $h$ of $S^1$ so that $h(K)$ is strictly
contained in $K$ (or $h(K)$ strictly contains $K$).

\begin{prop}
\label{prop:contraexemplo}
 Let $G$ and $K$ be as in Proposition \ref{transnonmin} and
$f_1,\ldots,f_n$ be a symmetric system of generators of $G$.
Consider any  homeomorphisms $h$ of $S^1$ such that $h(K)$
strictly contains $K$. Then the IFS generated by
$f_1,\ldots,f_n,h$ is forward minimal but not backward minimal.
\end{prop}
\begin{proof}
Let  $ K_1\eqdef h(K)$. By assumption $K \subsetneq K_1$. We claim
that the IFS generated by $f_1,\ldots,f_n,h$ is forward minimal
but not backward minimal.
\begin{itemize}[itemsep=0.1cm,label=\textbullet, leftmargin=1pc, labelsep=*]
\item \emph{The IFS is not backward minimal}: since $K$ is invariant
by the
group $G$,
$$\text{$f_i^{-1}(K)=K$  \ \ for $i=1,\ldots,n$.}
$$ We also have
$h^{-1}(K)\subset h^{-1}(K_1)=K$. Thus, $K$ is forward invariant
by $f_1^{-1},\ldots,f_n^{-1},h^{-1}$ and so the IFS is not
backward minimal.
\item \emph{The IFS is forward minimal}: let $B\subset S^1$ be a forward invariant by $f_1,\ldots,f_n,h$ closed set. In particular $B$
is invariant by $G$, hence $B\in\{\emptyset,K,S^1\}$ by Lemma
\ref{3invariant}. Moreover $B\not= K$ since $K$ is not invariant
by $h$ (otherwise $K_1=h(K)=h(B)\subset B=K$ but $K_1$ strictly
contains $K$). So, $B\in \{\emptyset,S^1\}$, which means that the
IFS is forward minimal. \qedhere
\end{itemize}
\end{proof}

\noindent {\sf b) Strict attractors:}  To complete the proof of
Corollary~\ref{maincor-contraexemplo} we need to show that $S^1$
is a strict attractor of the IFS generated by $f_1,\dots,f_n,h$ in
Proposition~\ref{prop:contraexemplo}. We infer this from the next
result.

We say that an IFS generated by a family $\mathscr{F}$ of
continuous maps of $X$ is \emph{quasi-symmetric} if there is $f\in
\mathscr{F}$ so that it inverse map $f^{-1}\in \mathscr{F}$.

\begin{prop}
\label{prop-quasi-symmetric} Consider a minimal quasi-symmetric
IFS on a compact connected Hausdorff space $X$. Then $X$ is an
strict attractor of this IFS.
\end{prop}

Before to prove the above proposition we need  the following Lemma
(c.f.~\cite[Lemma~4.15]{M15}). Again, for completeness,  we
include the proof.

\begin{lem}
\label{lem-minimal-k} Consider a minimal IFS generated by a family
$\mathscr{F}$ of continuous maps of a connected  Hausdorff
topological space $X$. Then the IFS generated by
$\mathscr{F}^2=\{f\circ g : f, g \in  \mathscr{F} \}$is also
minimal.
\end{lem}
\begin{proof}
Throughout the proof, we extend the Hutchinson operator
$F=F_\mathscr{F}$ to the hyperspace of non-empty closed sets. We
want to prove that if $B$ be a non-empty closed subset of $X$ so
that $F^2(B)\subset B$ then $B=X$. Notice that
$$
B'\eqdef B\cup F(B) \quad \text{and} \quad B''\eqdef B\cap F(B)
$$
are both forward invariant set, i.e. $F(B')\subset B'$ and
$F(B'')\subset B''$. By the minimality of the IFS generated by
$\mathscr{F}$ it follows that $B'=X$. Hence, since $X$ is a
connected space and both $B$ and $F(B)$ are closed we get that
$B''\not =\emptyset$. Thus, again by the minimality we have that
$B''=X$ and therefore $B=X$.
%
%
%
%
\end{proof}

\begin{proof}[Proof of Proposition~\ref{prop-quasi-symmetric}]
Let $K$ be a compact set of $X$. We want to show that $F^n(K) \to
X$ in the Vietoris topology. The first observation is that, since
the IFS is quasi-symmetric then $K\subset F^2(K)$. Then,
$$
\text{$(F^2)^n(K) \subset (F^{2})^{n+1}(K)$ \ \ for all $n\geq 1$.}
$$
By Lemma~\ref{lem-minimal-k}, the IFS generated by $\mathscr{F}^2$
is also minimal and thus, for every open set $V$ of $X$ there is
$n_0\in\mathbb{N}$ such that $(F^2)^{n_0}(K)\cap V\not
=\emptyset$. So, by the monotonicity of this sequence,
$(F^{2})^n(K)\cap V \not=\emptyset$ for all $n\geq n_0$. Thus,
according to Lemma~\ref{lem-conv-Vietoris}, we have that
$(F^2)^n(K)\to X$ in the Vietoris topology. By means of the
continuity of the Hutchinson operator $F$ and since $X$ is a
self-similar set we also have that $F^{2n+1}(K)\to X$. Thus, we
conclude that $F^n(K)\to X$.
\end{proof}

We want to remark that in the case of the homeomorphisms of the
circle we have an stronger result:

\begin{pro}
\label{prop:Malicet} Let $f_1,\ldots f_k$ homeomorphisms of $S^1$
without a common invariant probability measure, and such that the
IFS generated by them is minimal. Then $S^1$ is a strict
attractor~for~this~IFS.
\end{pro}
\begin{proof}
Let $x$ in $S^1$, and let $\mu_n$ be the law of
$f_\omega^n(x)=f_{\omega_n}\circ\cdots\circ f_{\omega_1}(x)$,
where $\omega_1,\ldots,\omega_n$ are chosen independently and
uniformly on $\{1,\ldots,k\}$. By \cite[Cor.~2.6]{M15}, the
sequence $(\mu_n)_{n\in\mathbb{N}}$ converges weakly as $n\to
\infty$ to the unique stationary probability measure $\mu$ of the
system, i.e. to the self-similar measure
$$\mu=\frac{1}{k}((f_1)_*\mu+\cdots+(f_k)_*\mu).$$ Moreover, this measure
$\mu$ has total support because its topological support is
invariant by $f_1,\ldots,f_k$. Consequently, for any interval $I$
of the circle, $\mu(I)>0$ and so $\mu_n(I)>0$ for $n$ large
enough. Since we clearly have that $\mathrm{supp}(\mu_n)\subset
F^n(\{x\})$,  we deduce that $F^n(\{x\})\cap I\not=\emptyset$ for
all $n$ sufficient large, and hence we conclude by
Lemma~\ref{lem-conv-Vietoris} that $S^1$ is an atractor.
\end{proof}

\subsection{Sufficient conditions}
\label{sec44}

In what follows, $A$ denotes a quasi-attractor.

\subsubsection{Well-fibred attractors}
We start studying the relation between strongly-fibred and
well-fibred quasi-attractors.

\begin{pro}
\label{prop-strong} If $A$ is strongly-fibred then it is
well-fibred. Moreover, if in addition $A$ is a strict attractor
then for every compact set $K$ in $\mathcal{B}(A)$ and every open
set $U$ so that $A\cap U\not=\emptyset$ there exists $g \subset
\Gamma$ such that $g(K)\subset U$.
\end{pro}

\begin{proof}
Consider a compact set $K$ in $A$ and let $U$ be any open set such
that $A\cap U\not=\emptyset$. Since $A$ is strongly-fibred, we get
$\omega\in\Omega$ such that
$$
A_{\omega}= \bigcap_{n=1}^\infty f_{\omega_1}\circ\cdots \circ
f_{\omega_n}(A) \subset U.
$$
Notice that since $f_i(A)\subset A$ for $i=1,\dots,k$ then
$f_{\omega_1}\circ\cdots \circ f_{\omega_n}(A)$ is a nested
sequence of compact sets and thus, for $n$ large enough,
$f_{\omega_1}\circ\cdots \circ f_{\omega_n}(A) \subset U$. In
particular, taking $h= f_{\omega_1}\circ\cdots \circ f_{\omega_n}
\in \Gamma$ we have that $h(K)\subset U$. This proves that $A$ is
well-fibred.

We will assume now that $A$ is a strict attractor and consider $K$
in $\mathcal{B}(A)$.  As above we have that $h(A) \subset U$. We
claim that there exists a neighborhood $V$ of $A$ such that $h(V)
\subset U$. Indeed, it suffices to note that $h$ is a continuous
map and hence $h^{-1}(U)$ is an open set containing the compact
set $A$.
%
%
%
Since $A$ is a strict attractor, $F^n(K)\to A$ in the Vietoris
topology and in particular, there is $f\in \Gamma$ such that
$f(K)\subset V$. Thus, taking $g=h\circ f \in \Gamma$, it follows
that  $g(K) \subset h(V) \subset U$.
\end{proof}

\begin{rem}
If $A$ is strongly fibred we have proved that one can contract any
compact set in $A$. In particular we can contract $A$ and this
implies that there exists some generator $f_i$ such that
$f_i(A)\not=A$.
\end{rem}

Now, we give an example of an IFS defined on $S^1$ whose unique
strict attractor is the whole space (that is the IFS is minimal)
and it is well-fibred but not strongly-fibred. This example shows
that these two properties are not equivalent. See also
Corollary~\ref{cor-forward-backward} at the end of this
subsection.

\begin{example}
\label{example} Consider the IFS generated by two diffeomorphisms
$g_1, g_2$, where $g_1$ is rotation with irrational rotation
number and $g_2$ is an orientation preserving diffeomorphism with
a unique fixed point $p$ such that $Dg_2(p)=1$ and $\alpha$-limit
set and $\omega$-limit set of each point $q \in S^1$ is equal to
$\{p\}$.
Clearly, the IFS acts minimally on $S^1$ and have no common
invariant measure thus $A=S^1$ is the attractor. Since $g_1$ and
$g_2$ map $S^1$ onto itself, it follows that for each $\omega \in
\Omega$, the fiber $A_\omega =S^1$. This implies that $S^1$ is not
strongly-fibred, but it still well-fibred. Indeed, let $K$ be any
compact set so that $K\not= S^1$. Then, there is an open arc $J$
of $S^1$ which is not dense in $S^1$ such that $K\subset  J$. If
$J$ contains the fixed point $p$, there is an integer $n$ such
that $g_1^n(J)$ does not contain $p$. So, without less of
generality, we may assume that $p \not\in J$. Now, it is easy to
see that $g_2^k(J)$ tends to $p$ as $k \to \infty$. This implies
that $A=S^1$ is well-fibred.
\end{example}

Above example is based on the fact that $A$ satisfies that
$f_i(A)=A$ for all $i=1,\dots,k$. The above proposition and the
following show that if  $f_i(A)$ is not equals to $A$ for some
generator $f_i$ then both properties are equivalent.

\begin{pro}\label{equiv}
If $A$ is well-fibred and $f_i(A)\not = A$ for some
$i\in\{1,\dots,k\}$ then $A$ is strongly-fibred.
\end{pro}
\begin{proof}
First of all note that it suffices to prove that for any open set
$U$ with $U\cap A\not =\emptyset$, there is $h\in\Gamma$ so that
$h(A)\subset U$. To this end, notice that since $A$ is a
quasi-attractor then the action of $\Gamma$ restricted to $A$ is
minimal. Then, there exist $h_1,\dots,h_m\in \Gamma$ so that
$A\subset h_1^{-1}(U)\cup\dots\cup h_m^{-1}(U)$. On the other
hand, by assumption, there is $i \in \{1,\dots, k\}$ such that
$f_{i}(A)\neq A$. Hence $f_i(A)$ is a compact set strictly
contained in $A$ and since $A$ is well-fibred there exist $g\in
\Gamma$ and $j\in\{1,\dots,m\}$ such that $g(f_i(A)) \subset
h_j^{-1}(U)$. Thus, taking $h=h_j\circ g\circ f_i\in\Gamma$, it
follows that $h(A)\subset U$ concluding the proof.
\end{proof}

In order to proof Theorem~\ref{thmE}, we need a lemma (compare
with Claim~\ref{claim-desisdad-disj}). Here we understand
$f_{i_t}\circ \cdots \circ f_{i_1}$ for $t=0$ as the identity map.

\begin{lem}
\label{claim2-desisdad-disj}  If for any non-empty open set
$I\subset X$ with $A\cap I \not =\emptyset$, there exist a
neighborhood $Z$ of $A$ and $f_{i_s}\circ \cdots \circ f_{i_1} \in
\Gamma$  such that
\begin{equation*}
\label{eq:1} \text{for each $z\in Z$ there is $t \in
\{0,\dots,s\}$ so that $f_{i_t}\circ \cdots \circ f_{i_1}(z)\in
I$}
\end{equation*}
then
$$
A \subset \overline{O^+_\omega(x)} \quad \text{for all disjunctive
sequences $\omega\in \Omega$  and  $x\in \mathcal{B}^*_p(A)$.}
$$
\end{lem}
\begin{proof}
Consider $x\in \mathcal{B}^*_p(A)$, disjunctive sequence
$\omega\in\Omega$ and any open set $I$ such that $A\cap I \not =
\emptyset$. Being $Z$ a neighborhood of $A$ and $\Ls F^n(\{x\})=
A$ we can choose $m\geq 1$ such that
$$
[\sigma^m(\omega)]_j = i_j \ \ \text{for $j = 1,\dots, s$  \ \ and
\ \ $z=f^{m}_\omega(x)\in Z$.}
$$
Hence, by assumption, there exists $t=t(z)$ such one has that
$f^{m+t}_\omega(x) \in I$ which proves the density on $A$ of the
$\omega$-fiberwise orbit of $x$.
\end{proof}

The following result proves the first part in Theorem~\ref{thmE}.

\begin{pro}
A well-fibred quasi-attractor $A$  is renderable by the
deterministic chaos game.
\end{pro}
\begin{proof}
In order to apply Lemma~\ref{claim2-desisdad-disj}, we consider
any non-empty open set $I$ with $I_A\eqdef A\cap I \not
=\emptyset$.
Hence, $K=A\setminus I_A$ is a compact set so that $K\not=A$.
Since $A$ is a quasi-attractor, the action of $\Gamma$ restrict to
$A$ is minimal and thus, there exist $h_1,\dots,h_m \in \Gamma$
such that $A\subset h_1^{-1}(I)\cup \dots \cup h^{-1}_m(I)$. On
the other hand, since $A$ is well-fibred, there exist $i\in
\{1,\dots,m\}$ and $g\in \Gamma$ such that $g(K) \subset
h_i^{-1}(I)$.  Take $h=h_i\circ g$. By continuity of the
generators we find an open set $U$ with $K\subset U$ such that
$h(U) \subset I$. Take, $Z=U\cup I$ and $f_{i_s}\circ\cdots \circ
f_{i_1}=h\in \Gamma$. Clearly, $Z$ is open with $A= K\cup I_A
\subset U \cup I =Z$ and for every $z\in Z$, there is $t\in
\{0,s\}$ such that $f_{i_t}\circ \cdots \circ f_{i_1}(z) \in I$.
Lemma~\ref{claim2-desisdad-disj} implies $A$ is renderable by the
deterministic chaos game.
\end{proof}

Now, we conclude the  proof of Theorem~\ref{thmE}.

\begin{pro}
Consider a well-fibred forward minimal IFS generated by continuous
maps of a compact Hausdorff topological space $A$. Assume the IFS
is either, strongly-fibred or invertible (its generators are
homeomorphisms). Then
$$
         \Omega \times A = \overline{\{\Phi^n(\omega,x):
         n\in\mathbb{N}\}} \quad \text{for all disjunctive sequences $\omega$ and  $x \in A$}.
$$
\end{pro}
\begin{proof}
Let $\omega \in \Omega$ be a disjunctive sequence and consider
$x\in A$. We want to show that $(\omega,x)$ has dense orbit in
$\Omega \times A$ under the skew-product $\Phi$. In order to prove
this, let $C^+_\alpha \times I$ be a basic open set of
$\Omega\times A$. That is, $C^+_\alpha$ is a cylinder in $\Omega$
around of a finite word $\alpha=\alpha_1\dots \alpha_\ell$ and $I$
is an open set in $A$. In fact, we can assume that $I$ is not
equal to the whole space. It suffices to prove that there exists
an iterated by $\Phi$ of $(\omega,x)$ that belongs to $C^+_\alpha
\times I$. To do this, similarly as in the previous proposition,
we use the forward minimality of $\Gamma$ on $A$ to find maps
$h_1,\dots, h_m\in\Gamma$ such that $A= h_1^{-1}(I)\cup \dots \cup
h^{-1}_m(I)$. Set $f=f_{\alpha_\ell}\circ \cdots \circ
f_{\alpha_1} \in \Gamma$.

Assume first that the IFS is strongly-fibred. Then there exists a
generator $f_i$ such that $K=f_i(A) \not =A$. By
Proposition~\ref{prop-strong}, the IFS is well-fibred and thus we
find $g\in \Gamma$ such that $g(K)\subset h_\ell^{-1}(I)$ for some
$\ell\in\{1,\dots,m\}$. Hence, $h(K)\subset I$ where
$h=h_\ell\circ g$. Let $f\circ h\circ f_i =f_{i_s}\circ\cdots
\circ f_{i_1}$. Since $\omega\in \Omega$ is a disjunctive
sequence, we can choose $m\geq 1$ such that
\begin{equation}
\label{eq:sigma}
 [\sigma^m(\omega)]_j = i_j \ \ \text{for $j =
1,\dots, s$.}
\end{equation}
Set $z=f^{m}_\omega(x)$. Then, $h \circ f_i(z)\in I$. Moreover,
$$
\Phi^{m+t}(\omega,x)=\Phi^t(\sigma^m(\omega),z)=
(\sigma^{m+t}(\omega),h\circ f_i(z))\in C^+_\alpha\times I
$$
where $t=1+|h|$ being $|h|$ the length of $h$ with respect to
$\mathscr{F}=\{f_1,\dots,f_k\}$.

Now, assume that the well-fibred forward minimal IFS is also
invertible. Hence $f$ is a homeomorphism of $A$ and thus
$\emptyset \not =f(I) \not = A$ is an open set. Let $K=A\setminus
f(I)$. Notice that $K$ is a non-empty compact set different of $A$
and by means of the ``contractibility'' of the IFS we get $g\in
\Gamma$ so that $h(K)\subset I$ where $h=h_\ell\circ g$ for some
$\ell\in \{1,\dots,m\}$. Let $f\circ h \circ f=f_{i_s}\circ\cdots
\circ f_{i_1}$. Similar as above, since $\omega$  is a disjunctive
sequence we choose $m\geq 1$ satisfying~\eqref{eq:sigma} and
denote $z=f^m_\omega(x)$. If $z\in I$,
$\Phi^m(\omega,x)=(\sigma^m(\omega),z) \in C_\alpha^+ \times I$.
Otherwise, $f(z)\in K$ and then $h\circ f(z) \in I$ and thus
$$
\Phi^{m+t}(\omega,x)=\Phi^t(\sigma^m(\omega),z)=
(\sigma^{m+t}(\omega),h\circ f(z))\in C^+_\alpha\times I
$$
where $t=|f|+|h|$ being $|f|$ and $|h|$ the length of $f$ and $h$
respectively.
\end{proof}

We end this subsection showing a broad family of IFSs with a
well-fibred quasi-attractor which are not strongly-fibred. Notice
that this family contains the IFS of Example~\ref{example}.

\begin{cor}
\label{cor-forward-backward}
 Consider a forward and backward
minimal IFS of homeomorphisms of a metric space $X$ and assume
that there is a map $h$ in the semigroup $\Gamma$ generated by
these maps with exactly two fixed points, one attracting and one
repelling. Then $X$ is a well-fibred quasi-attractor and
consequently is renderable by the deterministic chaos game.
\end{cor}
\begin{proof}
The forward minimality implies that $X$ is a quasi-attractor.
Consider now any compact set $ K \subset X$ such that $K \not =
X$. By the backward minimality there exist $T_1,\dots, T_s \in
\Gamma$ such that
$$
    X = \bigcup_{i=1}^s T_i(X\setminus K).
$$
Let $p$ and $q$ be, respectively, the attracting and the repelling
fixed points of $h$.  Then there is $i\in\{1,\dots,s\}$ so that
$q\in T_i (X\setminus K)$.  Therefore, $q \not\in T_i(K)$ and then
the diameter of $h^n\circ T_i(U)$ converges to zero. This shows
that the action is well-fibred and completes the proof.
\end{proof}

\subsubsection{Quasi-attractors of symmetric IFSs}
We will proof Theorem~\ref{thmF}. This theorem
extends~\cite[Thm.~3.3]{F02} for compact Hausdorff topological
spaces.

\begin{prop}
If $A$ is a quasi-attractor of a symmetric IFS on a Hausdorff
topological space then it is renderable by the deterministic chaos
game.
\end{prop}

\begin{proof}
We will use Lemma~\ref{claim2-desisdad-disj}. To accomplish this,
let $I$ be an open set so that $A\cap I \not=\emptyset$. By the
minimality of the action of $\Gamma$ restricted to $A$, there are
$h_1,\dots,h_m \in \Gamma$ so that $A \subset h_1^{-1}(I)\cup
\dots \cup h^{-1}_m(I)$. Set $Z$ be the union of these open sets.
Since the IFS has a symmetric system of generators
$\mathscr{F}=\{f_1,\dots,f_k\}$,  we can write
$$
f_{i_s}\circ\cdots\circ f_{i_1}=h_m^{-1}\circ h_m\circ \cdots\circ
h_2^{-1}\circ h_2\circ h_1^{-1} \circ h_1.
$$
Hence,
$$
    Z \subset \bigcup_{j=1}^t  f^{-1}_{i_1}\circ \cdots
    \circ f^{-1}_{i_j}(I).
$$
This implies that for each $z\in Z$, there is $t\in \{1,\dots,s\}$
so that $f_{i_t}\circ \cdots \circ f_{i_1}(z) \in I$. Thus, by
Lemma~\ref{claim2-desisdad-disj}, $A$ is renderable by the
deterministic chaos game.
\end{proof}

To end this subsection, we give an example of a quasi-attractor of
an symmetric IFS on the torus $\mathbb{T}^2$ which is not neither
well-fibred nor a quasi-attractor of a non-expansive IFS nor have
a minimal map.

\begin{example}
\label{example-final} Let $f:\mathbb{T}^2 \to \mathbb{T}^2$ be a
generalized north-south pole diffeomorphism on the torus
$\mathbb{T}^2$. By this we mean that the non-wandering set of $f$,
$\Omega(f)$, consists of one fixed source, $q$, one fixed sink,
$p$, and saddle type periodic orbits. Let $S$ be the set of all
the saddle type periodic points of $f$. For simplicity we assume
that $S$ consists of two saddle points so that
$\mathcal{W}=W^s(S)\cup W^u(S) \cup \{p,q\}$ consists of forth
circles: two disjoint circles following the meridian direction
other two disjoint circles following the parallel directions. For
every $x\in \mathbb{T}^2\setminus \mathcal{W}$ it holds that
$f^n(x)\to p$ and $f^{-n}(x)\to q$ as $n\to \infty$. On the other
hand, consider a translation $R_\lambda: \mathbb{T}^2\to
\mathbb{T}^2$, $R_\lambda(x,y)=(x+\lambda_1,y+\lambda_2)$, where
$\lambda=(\lambda_1,\lambda_2)$ is an irrational vector, i.e.,
$\lambda_i\in \mathbb{R}\setminus \mathbb{Q}$ for $i=1,2$. Since
the IFS generated by $f,f^{-1}, R_\lambda,R_\lambda^{-1}$ on
$\mathbb{T}^2$ has minimal elements,  according
to~\cite[Prop.~1]{BFS14}, it is renderable by the deterministic
chaos game. Moreover, it is not difficult to see that this IFS is
$C^1$-robustly minimal, i.e., the minimality persists under small
$C^1$-perturbations on the generators (indeed, easily one can
construct a ``blending region'' around the attracting fixed point
and then apply~\cite[Thm.~6.3]{BFS16}). Thus, there is a rational
vector $\alpha$  closed to $\lambda$ so that the IFS generated by
$f,f^{-1}, R_\alpha,R_\alpha^{-1}$ acts minimally on
$\mathbb{T}^2$. By Theorem \ref{thmF}, this IFS is renderable by
the deterministic chaos game. Clearly, it does not contain any
minimal element and it is not a non-expansive IFS. Also, it is not
well-fibred (indeed, it suffices to consider a compact
neighborhood of a circle that contains $p$ and the unstable
manifold of one saddle).
\end{example}

\appendix
\renewcommand{\thesection}{A}
\renewcommand{\theequation}{A.\arabic{equation}}
\setcounter{equation}{0} \setcounter{thm}{0}
\section{}
\label{s:appendix} In this appendix we extend the results due to
Bransley, Le\'sniak and Rypka,~\cite{BLR15,L14} on the
probabilistic and the deterministic chaos game for attractors of
IFSs to the general case of quasi-attractors. The proofs basically
follow the same ideas of~\cite{BLR15,L14}  with some minor
modifications and improvements.

\vspace{0.2cm} {\sf On Bransley, Le\'sniak and Rypka probabilistic
chaos game for quasi-attractors:}

\begin{thm*}
Every first-countable quasi-attractor of an IFS of continuous maps
of a Hausdorff topological space is renderable by the
probabilistic chaos game.
\end{thm*}
\begin{proof}
Let $x_0$ be a point of $\mathcal{B}^*_p(A)$ and let $U$ an open
subset of $A$. We want to prove that the event
$$
E=E(x_0,U)\eqdef \{\omega\in\Omega: \, O^+_\omega(x_0)\cap
U\not=\emptyset\}
$$
has probability $1$.  Since $A$ is a minimal forward invariant
set, for any $x$ in $A$ we can find a finite sequence $i_1,\ldots
,i_m$ such that $f_{i_m}\circ\cdots\circ f_{i_1}(x)$ belongs to
$U$. Then, by using the compactness of $A$ and the continuity of
the generators we can actually find an integer $m_0$ and functions
$i_1,\ldots,i_{m_0}$ from a neighborhood $V$ of $A$ into
$\{1,\ldots k\}$ such that for every $x\in A$, there is $m\leq
m_0$ such that $f_{i_m(x)}\circ\cdots\circ f_{i_1(x)}(x)\in U$.
Since $\Ls F^n(\{x_0\}) \subset A$ then $x_n=f_\omega^n(x_0)\in V$
for all $n$ large enough. Then $E_n=\{\omega \in \Omega: \,
\omega_{n+m_0}=i_{m_0}(x_n),\ldots, \omega_{n+1}=i_1(x_n)\}$ is
obviously contained in $E$, and since the random variable $x_n$
depends only on $\omega_1,\ldots,\omega_n$, we obtain from
\eqref{conditional assumption} that
$$\mathbb{P}(E\,|\, \omega_1,\ldots,\omega_n)\geq \mathbb{P}(E_n\,|\, \omega_1,\ldots,\omega_n)\geq p^{m_0}\eqdef \delta_0.$$
In particular, for every set $C$ in the $\sigma$-algebra generated
by $\omega_1,\ldots,\omega_n$, we have the inequality
$\mathbb{P}(E\cap C)\geq \delta_0 \, \mathbb{P}(C)$, and since $n$
is arbitrary, this inequality actually holds for any Borel set $C$
of $\Omega$. Choosing $C=\Omega\setminus E$, we deduce that
$\mathbb{P}(\Omega\setminus E)=0$  concluding that $E$ has full
probability.

Finally, we will prove that with probability $1$,
$O^+_\omega(x_0)$ is dense in $A$. First notice that any
quasi-attractor is a separable set. Hence, let us choose
$(z_i)_{i\in\mathbb{N}}$ a sequence dense in $A$ and for each $i$,
$(U_{i,j})_{j\in\mathbb{N}}$ a basis of neighborhood of $z_i$. Let
$$
\Omega(x_0)=\bigcap_{i\in \mathbb{N}}\bigcap_{j\in\mathbb{N}}
E(x_0,U_{i,j}).
$$
From the above, $\Omega(x_0)$ has full probability.  Given any
open set $U$ of $X$ so that $U\cap A\not=\emptyset$. Then we find
$U_{i,j}$ so that $U_{i,j}\subset U$.  Then, for every $\omega \in
\Omega(x_0)$, the set $O_\omega^+(x_0)$ intersects  $U_{i,j}$ and
in particular $U$. Thus $O^+_\omega(x_0)$ is dense in $A$.

\end{proof} {\sf On Le\'sniak deterministic chaos game for quasi-attractor
of non-expansive IFS:}

\begin{thm*}
Every quasi-attractor of a non-expansive IFS on a metric space is
renderable by the deterministic  chaos game.
\end{thm*}
\begin{proof}
We will use Lemma~\ref{claim2-desisdad-disj}. To accomplish this,
let $I$ be an open set so that $A\cap I \not=\emptyset$. We can
suppose that $I=B_{2\varepsilon}(y_0)$ is an open ball of radius
$2\varepsilon>0$ and centered at $y_0\in A$. By the compactness of
$A$, we can find $y_1,\dots,y_m\in A$ such that $ A \subset
B_{\varepsilon}(y_0)\cup B_{\varepsilon}(y_1)\cup \dots \cup
B_{\varepsilon}(y_m) \eqdef Z$. Being the action of  $\Gamma$ on
$A$ minimal, we find $h_1\in \Gamma$ such that $h_1(y_1)\in
B_{\varepsilon}(y_0)$. Recursively, constructed $h_{i-1}$ we find
$h_i\in \Gamma$ such that $h_i\circ \cdots \circ h_1(y_i) \in
B_{\varepsilon}(y_0)$. On the other hand, for each $z\in Z$, there
is $i\in \{0,\dots,m\}$ such that $d(z,y_i)\leq \varepsilon$.
Since the IFS is non-expansive,
$$
    d(h_{i}\circ \dots \circ h_1(z), h_i\circ \cdots \circ
    h_1(y_i)) \leq d(z,y_i)\leq \varepsilon.
$$
Since $d(h_i\circ \dots \circ
    h_1(y_i), y_0) \leq \varepsilon$ it follows that $d(h_{i}\circ \cdots
\circ h_1(z), y_0)\leq 2\varepsilon$. That is, $h_{i}\circ \cdots
\circ h_1(z) \in I$ for some $i\in \{0,\dots,m\}$ where we recall
that $h_{i}\circ \cdots \circ h_1$ for $i=0$ denotes the identity
map. Hence writing
$$
    f_{i_s}\circ \cdots \circ f_{i_1} =  h_m\circ \cdots \circ h_1
    \quad \text{where $f_{i_j} \in \mathscr{F}$}
$$
we have obtained that there is $t\in \{0,\dots,s\}$ so that
$f_{i_t}\circ \cdots \circ f_{i_1}(z) \in I$. Thus, by
Lemma~\ref{claim-desisdad-disj}, $A$ is renderable by the
deterministic chaos game.
\end{proof}


\end{document}